\documentclass[letterpaper, 10pt, conference, final, twocolumn]{ieeeconf}      % Use this line for a4

\usepackage{epsfig} %% for loading postscript figures

\usepackage{amsmath,graphicx,amsfonts,amssymb,epsfig,subfig,mathrsfs,mathtools}
\allowdisplaybreaks
\usepackage{ntheorem}
\usepackage{cuted}
\usepackage{arydshln}
\usepackage{blkarray}

\usepackage{color}
\usepackage{lipsum}
\usepackage{rotating}
\usepackage{graphicx}%
\usepackage{algorithm}

\usepackage[utf8]{inputenc}
\usepackage{xspace}
\usepackage{algorithmicx}
\usepackage{algpseudocode}
\usepackage{cite}
\usepackage{tikz}
\usepackage{cancel}
\usepackage{textcomp}
 \definecolor{lin}{RGB}{240,0,0}	
 \definecolor{paleblue}{RGB}{0,9,255}
\usepackage[colorlinks=true,linkcolor=lin,citecolor=paleblue,pdfa]{hyperref}

\newcommand{\map}[3]{#1: #2 \rightarrow #3}
\newcommand{\setdef}[2]{\{#1 \; | \; #2\}}
\newcommand{\st}{\ensuremath{\operatorname{s.t.}}}
\newcommand{\real}{\ensuremath{\mathbb{R}}}
\newcommand{\prob}{\ensuremath{\mathbb{P}}}
\newcommand{\realnonnegative}{\ensuremath{\mathbb{R}}_{\ge 0}}

\newcommand{\until}[1]{\{0,1,\dots, #1\}}
\newcommand{\untilone}[1]{\{1,\dots, #1\}}
\newcommand{\subscr}[2]{#1_{\textup{#2}}}
\newcommand{\supscr}[2]{#1^{\textup{#2}}}
\newcommand{\vect}[1]{\mathbf{#1}}
\newcommand{\vectorones}[1]{\vect{1}_{#1}}
\newcommand{\vectorzeros}[1]{\vect{0}_{#1}}
\newcommand{\Norm}[1]{\|#1\|}
\newcommand{\trans}[1]{{#1}^\top}
\newcommand{\obj}{\ensuremath{\operatorname{obj}}}

\newcommand{\dom}{\ensuremath{\operatorname{dom}}}

\newcommand{\UB}{\ensuremath{\operatorname{UB}}}
\newcommand{\LB}{\ensuremath{\operatorname{LB}}}
\newcommand{\graph}{\mathcal{G}}
\newcommand{\nodes}{\mathcal{V}}
\newcommand{\edges}{\mathcal{E}}

\DeclareMathOperator*{\argmax}{argmax}

\makeatletter
\newtheoremstyle{breaknote}%
  {\item{\theorem@headerfont
          ##1\ ##2\theorem@separator}\hskip\labelsep\relax}%
  {\item{\theorem@headerfont
          ##1\ ##2\ (##3)\theorem@separator}\hskip\labelsep\relax}
\newcommand{\leqnomode}{\tagsleft@true}
\newcommand{\reqnomode}{\tagsleft@false}					
\makeatother
\reqnomode
\theoremstyle{breaknote}
\newtheorem{assumption}{Assumption}[section]

\newtheorem{lemma}{Lemma}[section]
\newtheorem{proposition}{Proposition}[section]
\newtheorem{remark}{Remark}[section]

\allowdisplaybreaks
\renewcommand{\baselinestretch}{.977}

\title{Data-driven Variable Speed Limit Design for Highways \\via
  Distributionally Robust Optimization}

\author{D. Li, D. Fooladivanda and S. Mart{\'\i}nez}

\begin{document}

\maketitle

%%%%%%%%%%%%%%%%%%%%%%%%%%%%%%%%%%%%%%%%%%%%%%%%%%%%%%%%%%%%%%%%%%%%%%
\begin{abstract}
  This paper introduces an optimization problem (P) and a solution
  strategy to design variable-speed-limit controls for a highway that
  is subject to traffic congestion and uncertain vehicle arrival and
  departure.  By employing a finite data-set of samples of the
  uncertain variables, we aim to find a data-driven solution that has
  a guaranteed out-of-sample performance. In principle, such
  formulation leads to an intractable problem (P) as the distribution
  of the uncertainty variable is unknown. By adopting a
  distributionally robust optimization approach, this work presents a
  tractable reformulation of (P) and an efficient algorithm that
  provides a suboptimal solution that retains the out-of-sample
  performance guarantee. A simulation illustrates the effectiveness of
  this method.
\end{abstract}

\section{Introduction}
Transportation networks constitute one of the most critical
infrastructure sectors today, with a major impact on the economics,
security, public health, and safety of a community. In these networks,
the accessibility of routes between increasingly-larger geographical
locations is highly dependent on the network connectivity as well as
on the traffic congestion on the available roads.  New advances on
smart infrastructure, computation, and communication make possible the
collection of real-time traffic data as well as the implementation of
novel control policies that can alleviate traffic problems. Motivated
by this, we consider a problem of traffic congestion reduction via
variable speed-limits and the assimilation of traffic data.

\textit{Literature Review:} Several congestion control schemes have been proposed in the
literature with the goal of mitigating congestion, such as ramp
metering control~\cite{SM-RC-LJ:17,BF-HD-GG-BA:17}, lane
assignment~\cite{RC-BLN-PM:17,GE-MS:18}, optimal
control~\cite{JS-SK:18,JA-PI-PM-DS:17}, logic-based
control~\cite{CS-GEA-AM_BC:16} and many other innovative control
strategies~\cite{SM-LJ:18,WC-BA-MA:18,CG-LE-SK:16}.  More recently,
variable speed limits have been proposed as an effective congestion
control mechanism in
transportation~\cite{FS-IM-MS-MM:17,AYY-MR-MAD:17,HY-HA-etl:17,IP-MP-IS:18}. Such
works exploit the Cell Transmission Model to capture the deterministic
distribution of traffic densities along a
road~\cite{CFD:94,SM-GP-GA-LJ:15}. In practice, these approaches may
be limited, due to the uncertainty on traffic density subject to
unknown actions by various drivers as well as vehicle arrival and
departure. However, the wide availability of data in real
time~\cite{WD-BS-etl:10,JCH-DBW-RH-XJB-QJ-AMB:10} can help reduce this
uncertainty and opens the way to the application of novel data-driven
optimization methods for control.  In this way, we consider here a
distributionally robust optimization (DRO)
framework~\cite{AC-JC:17,DL-SM:18-extended,RG-AJK:16,PME-DK:17} for
data assimilation. DRO uses finite data to make decisions with
desirable out-of-sample performance guarantees, and as such, it paves
the way for real-time decisions to dynamical transportation systems.
Here, we aim to answer two questions; that is, 1) what role variable
speed limits play in congestion, and 2) find an efficient approach for the
computation of data-driven variable speed limit controls with
performance guarantees.

\textit{Statement of Contributions:} In this work, we consider a
highway divided into equal-size segments. Vehicle arrival and
departure into each segment represent inflow and outflow disturbances
to traffic, and we model these disturbances as unknown stochastic
processes. Further, we assume that finite realizations of such random
variables can be acquired in real time and that a transportation
network operator can prescribe variable speed limits to control
congestion on each of the segments. In this setting, we propose a
novel data-driven variable-speed-limit control to limit congestion and
maximize the throughput of the road.
To do this, we first leverage the effect of variable speed limits to
limit traffic congestion. This is achieved by exploiting
approximations of the well-known Fundamental Diagram for various speed
limits. To ensure the performance of a data-driven solution with a
given confidence, we generalize the DRO framework in the literature to
handle the dynamical system constraints of our control
problem. 
Specifically, we define ambiguity sets, or the sets of system
trajectory distributions, to contain the distribution of the true
system trajectory with high probability. The proposed DRO approach
then allows us to obtain a set of speed limits with an out-of-sample
performance bound defined as the optimal objective value of a
worst-case optimization problem over the ambiguity set. As the
resulting problem is infinite-dimensional and intractable, we further
obtain an equivalent reformulation that reduces it into a
finite-dimensional problem.
Still the resulting problem is nonconvex, and our third contribution
provides an integer-solution search algorithm to find feasible
data-driven variable speed limits. This algorithm is based on the
decomposition of the nonconvex problem into mix-integer linear programs
and, as such, has certain convergence properties
guarantees. 
We establish that this solution procedure guarantees a feasible
solution with the out-of-sample performance guarantee with high
probability. We finally illustrate the performance of the proposed
algorithm in simulation.

\section{Preliminaries}\label{sec:pre}
Let $\real^{m\times n}$ denote the $m \times n$-dimensional real
vector space, and let the shorthand notations $\vectorones{m}$ and
$\vectorzeros{m}$ denote the column vector $\trans{(1,\cdots,1)} \in
\real^m$ and $\trans{(0,\cdots,0)} \in \real^m$, respectively.  Any
letter $x$ may have appended the following indices and arguments: it
may have the subscript $x_e$, with $e \in \mathbb{N}$, the argument
$x_e(t)$, $t \in \real$, and further a superscript $l\in \mathbb{N}$
as in $x^{(l)}_e(t)$. We assume that the dimension of the letter with
the most indexes belongs to $\real$, while their removal increases its
dimension. 
In this way, given $x^{(l)}_e(t) \in \real$, for several $e, t,$ and
$l$, we denote $x^{(l)}(t):=(x^{(l)}_1(t), x^{(l)}_2(t),\ldots)$, then
further $x^{(l)}:=(x^{(l)}(1), x^{(l)}(2),\ldots)$, and finally
$x:=(x^{(1)}, x^{(2)},\ldots)$.  The inner and component-wise products
of any two vectors $x,y \in \real^m$ are denoted by $\left\langle
  x,\;y\right\rangle$ and $x \circ y$, respectively. In addition, the
Kronecker product of any two vectors $x, y$ with arbitrary dimension
is denoted by $x \otimes y$. The $1$-norm of the vector ${x} \in
\real^m$ is denoted by $\Norm{{x}}$ and its dual norm is denoted by
$\Norm{{x}}_{\star}:=\sup_{\Norm{z}\leq 1}\left\langle
  z,\;x\right\rangle$. We have that $\|x\|_{**} = \|x\|$.

Let $X\subseteq \real^n$ be a subspace and let $X^{\star}$ denote the
dual space of $X$. For each $x \in X$, the dual $x^\star \in X^\star$
is defined as $x^\star(y) = \langle x, y \rangle$, for any $y\in
X$. Let $\map{f}{X}{\real}$ be a function on $X$ and we define its domain
of interest by $\dom f:=\setdef{x\in X}{-\infty<f(x)<+\infty}$. We say
$f$ is convex on $X$, if $f(\lambda x+(1-\lambda)y) \leq \lambda
f(x)+(1-\lambda)f(y)$ for all $x,y\in X$ and $0 \leq \lambda\leq
1$. 
We call $f$ lower semi-continuous on $X$, if $f(x) \leq
\liminf_{y\rightarrow x}f(y)$ for all $x \in X$. A function $f$ is
lower semi-continuous on $X$ if and only if its sublevel sets
$\setdef{x\in X}{f(x)\leq \gamma}$ are closed for each $\gamma \in
\real$.
We denote the convex conjugate of $f$ by
$\map{f^{\star}}{X}{\real\cup \{+\infty \}}$, which is defined
as $f^{\star}(x):=\sup_{y\in X} \left\langle
  x,\;y\right\rangle
-f(y)$. 
Let $f$ and $g$ denote two functions on $X$. The infimal convolution
of $f$ and $g$ is defined as $(f\square g)(x):= \inf_{y\in X}
f(x-y)+g(y)$.

Let $A$ be a set in $X$. We use the notion $\map{\chi_{A}}{X}{\real
  \cup \{+\infty \}}$ to denote the characteristic function of $A$,
i.e., $\chi_{A}(x)$ is equal to $0$ iff $x\in A$ and $+\infty$
otherwise. The support function of $A$ is defined as
$\map{\sigma_{A}}{X}{\real}$,
$\sigma_{A}(x):=\sup_{y\in A}\left\langle
  x,\;y\right\rangle$. It can be verified that
$\sigma_{A}(x)=[\chi_{A}]^{\star}(x)$ for all
$x\in X$, and $\chi_{A}$ is lower semi-continuous if
and only if $A$ is closed.
Let $f$ and $g$ denote two convex and lower semi-continuous functions
on $X$ with $\dom f \cap \dom g \ne \varnothing$. The conjugate of
$f+g$ has the following property:
$\left(f+g\right)^{\star}=(f^{\star}\square g^{\star})$. For more
information and details of these properties we refer readers
to~\cite[Theorem 11.23(a), Dual
operations]{RTR-RJBW:98} and references therein. 

\section{Problem Statement}\label{sec:ProbStat}
In this section, we first introduce the traffic model that we consider
in this paper, and then we propose a stochastic optimal control
framework to solve the proposed variable-speed limit design problem.

\subsection{Transportation System Model}
Consider a one-way road of length $L$, and divide the road into $n$
segments of equal size $L/n$. The road can be modeled as a chain
directed graph $\graph=(\nodes,\edges)$ with set of nodes
$\nodes:=\{0,1,\ldots,n\}$ and set of edges
$\edges:=\{(0,1),\cdots,(n-1,n)\}$. Each edge $e\in \edges$
corresponds to a road segment and each node $v\in \nodes$ corresponds
to a link between two road segments. We call node $0$ the source node,
and call node $n$ the sink node. Further, a node $v \in \nodes$ is
called an arrival node if there exists a non-zero inflow at node
$v$. Similarly, a node $v$ is called a departure node if there exists
a non-zero outflow at node $v$. Let $\nodes_{A}$ and $\nodes_{D}$
denote the set of arrival nodes and departure nodes, respectively. By
convention, we have node $0 \in \nodes_{A}$ and node $n \in
\nodes_{D}$.

Let us consider a time horizon of length $Q$, and assume that time is
divided into time slots of size $\delta$. Let $\mathcal{T}:=\until{T}$
denote the set of time slots, where $T=Q/\delta$. We denote by
$\overline{u}_e$, $\overline{\rho}_e$, and $\overline{f}_e$ the
maximal free flow speed, the jam density, and the capacity of edge
$e$, respectively.  Let $u_e \in \left[0, \overline{u}_e\right]$
denote the speed limit of the vehicles on edge $e$ and we consider
$u_e$ to be constant over the set of time slots $\mathcal{T}$. At each
time slot $t$, we denote by $\rho_e(t) \in [0, \overline{\rho}_e]$ the
density of the vehicles on edge $e$, and
the allowable flow rate on each edge $e$ by $f_e(t) \in [0,
\overline{f}_e]$, depending on the density $\rho_e(t)$ and speed limit
$u_e$ of the segment. Given a (constant) speed limit $u_e$, the
relationship between $f_e(t)$ and $\rho_e(t)$ can be characterized by
the fundamental diagram of edge $e\in \edges$, (e.g.,
see~\cite{NG-FCD:08,FS-IM-MS-MM:17}). This diagram determines the
nonlinear relationship $f_e(t):=f_e(\rho_e(t),u_e)$ between the
allowable flow rate $f_e(t)$ and the density $\rho_e(t)$ for a given
$u_e$, as shown in Fig.~\ref{fig:fd}.
In this study, we define the fundamental diagrams for various speed
limits as the curves shown in Fig.~\ref{fig:fd}. Note that for each
edge $e\in \edges$, the critical density $\rho^{c}_{e}({u}_{e})$ is
defined as the density at which the maximum allowable flow is
achievable for a given speed limit $u_e$. More precisely, given $u_e$,
the function $f_{e}(\rho,{u}_{e})$ is increasing if $\rho \leq
\rho^{c}_{e}({u}_{e})$ and decreasing if $\rho >
\rho^{c}_{e}({u}_{e})$.
\begin{figure}[tbp]%
\centering
\includegraphics[width=0.45\textwidth]{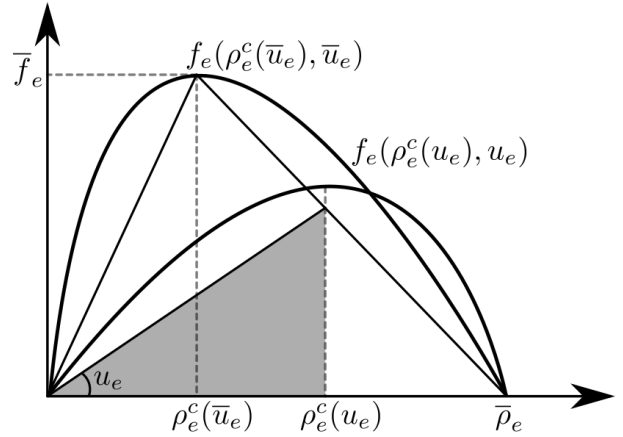}%
\caption{{\small Allowable flow-density relation of edge $e$ under
    speed limits $\overline{u}_e$ and $u_e$, respectively. The two
    curves are fundamental diagrams under the speed limits
    $\overline{u}_e$ and $u_e$. The straight lines are piecewise
    linear approximations of these two diagrams. The shaded region
    guarantees no congestion of edge $e$ under speed limit~$u_e$.} }%
\label{fig:fd}%
\end{figure}

Let $s_e(t) \in [0,\overline{u}_{e}]$ denote the average speed of the
vehicles on edge $e \in \edges$ at time $t \in \mathcal{T}$, and let
us assume that the majority of drivers have no incentive to exceed the
speed limit $u_e$, i.e., $s_e(t) \leq u_e$. Further, let $q_e(t)$
denote the flow rate on edge $e \in \edges$ at time slot $t \in
\mathcal{T}$. The flow $q_e(t)$ is equal to $\rho_e(t)s_e(t)$, for all
$e \in \edges$ and $t \in \mathcal{T}$. For each speed limit $u_e$,
edge $e \in \edges$ will be \textit{congested} at a time $t$ if the
flow rate $q_e(t)$ is greater than the allowable flow rate
$f_e(\rho_e,u_e)$ and the density $\rho^{c}_{e}(u_e)$ satisfies
$q_e(t) > f_e(\rho^{c}_{e}(u_e),u_e)$ and, thus, $\rho_e(t) >
\rho^{c}_{e}(u_e)$.  To prevent congestion at each time slot $t \in
\mathcal{T}$, we must have:
\begin{equation}
  \rho_e(t) \leq \rho^{c}_{e}(u_e),~\forall e\in\edges.
  \label{eq:RhoC}
\end{equation}
The above constraints are sufficient to guarantee that the road is not
congested, regardless of the flow rates.
  
The computation of critical density $\rho^{c}_{e}(u_e)$ and allowable
flow $f_e$ of each segment $e\in\mathcal{E}$ for different values of
speed limits $u_e$ is highly dependent on the fundamental diagram of
the segment. In this study, we approximate the fundamental diagram of
each segment with a finite set of piecewise linear functions as shown
in Fig.~\ref{fig:fd}. Each of these functions corresponds to a speed
limit. Let $\Gamma:=\{\gamma^{(1)}, \ldots, \gamma^{(m)}\}$ denote the
set of fixed non-zero speed limits.  For each edge $e\in \edges$ and
$u_e \in \Gamma$, we approximate the fundamental diagram of segment
$e$ by
\begin{equation*}
  f_{e}(\rho_{e}(t), u_e)=
	\begin{cases}
          u_e \rho_e(t) , & \; {\rm{if}} \; \rho_{e}(t) \leq \rho^{c}_{e}(u_e), \\
          \tau_e\overline{u}_e\left( \overline{\rho}_e - \rho_{e}(t)
          \right), & \; {\rm{o.w.},}
	\end{cases}
  \label{eq:f}
\end{equation*}
with
\begin{equation*}
  \rho^{c}_{e}(u_e):=\left({\tau_e\overline{\rho}_{e}\overline{u}_e}\right)/
  \left({\tau_e \overline{u}_{e} + {u}_{e}}\right),
\end{equation*}
where the parameter $\tau_e:=
{\overline{f}_e}/{\left(\overline{u}_e\overline{\rho}_e
    -\overline{f}_e\right)}$.

For each time $t\in\mathcal{T}$, let $\supscr{\omega_\mu}{in}(t)$
denote the random inflow of the starting node $\mu$ of the edge
$e\in\mathcal{E}$, let $\supscr{\omega_{\nu}}{out}$ denote the random
outflow of the ending node $\nu$ of the same edge, and let
$\omega_e(t):=\supscr{\omega_\mu}{in}(t)-\supscr{\omega_{\nu}}{out}(t)$
denote the difference between the inflow and outflow of edge $e$. In
this setting, each random variable $\supscr{\omega_\mu}{in}(t)$ has
nontrivial support $\mathcal{Z}_{\supscr{\omega_\mu}{in}(t)} \subset
\realnonnegative$ if $\mu \in \nodes_{A}$. Similarly, each random
variable $\supscr{\omega_\nu}{out}(t)$ has nontrivial support
$\mathcal{Z}_{\supscr{\omega_\nu}{out}(t)} \subset \realnonnegative$
if $\nu \in \nodes_{D}$. Without loss of generality, we assume that
the random inflows and outflows are independent from the speed limits
$u$.  Let $\vect{\rho}(0)=(\rho_1(0),\ldots,\rho_n(0))$ denote the
random initial density of the road $\mathcal{G}$ with nontrivial
support $\mathcal{Z}_{\vect{\rho}(0)} \subset \realnonnegative^n$, the
dynamics of the density on each edge $e \in \edges$ can be represented
by~\cite{CFD:94}
\begin{equation}
  \rho_e(t+1)= \rho_e(t)+ h(f_s(t)- f_e(t)+ \omega_e(t)), \; \forall \; t,
  \label{eq:CTM}
\end{equation}
where $h:=n\delta/L$ is determined by $n$, $\delta$ and $L$ such that
$h \leq 1/\max_{e\in \mathcal{E}}\{\overline{u}_e\}$, the subscript $s \in
\mathcal{E} \cup \varnothing$ denotes the preceding edge of edge $e$, and $f_e(t):=u_e\rho_e(t)
$ for each $e\in\mathcal{E}$, $t\in\mathcal{T}$. 

Random events, such as accidents on different segments of the road and
temporary lane closure, can affect the capacity and jam density of
each segment. In this study, we use $\supscr{f}{U}_e$ and
$\supscr{\rho}{U}_e$ to denote the temporary conditions on the
capacity and jam density. We assume that the system is in a specific
condition, and hence the values of the parameters $\supscr{f}{U}_e
\leq \overline{f}_e$ and $\supscr{\rho}{U}_e \leq \overline{\rho}_e$
are fixed and known to the operator. Since the values of
$\supscr{f}{U}_e$ and $\supscr{\rho}{U}_e$ are known, for all
$e\in\edges$, we can compute the maximum and minimum speed limits of
each segment under the certain event. For given $\supscr{f}{U}_e$ and
$\supscr{\rho}{U}_e$ of each segment $e$, we have $\rho^{c}_{e}(u_e)
u_e \leq \supscr{f}{U}_e$ and $\rho^{c}_{e}(u_e) \leq
\supscr{\rho}{U}_e-\pi$ with a small but positive threshold $\pi$ to
ensure non-zero flows on edge $e$. For each edge $e \in \edges$, we
need to ensure that the variable speed limit of the segment satisfies
the following constraint:
 \begin{equation}
\begin{aligned}
	\supscr{u}{L}_e \leq &u_e \leq \supscr{u}{U}_e, \; \\
	u_e \in \Gamma:=&\{\gamma^{(1)}, \ldots, \gamma^{(m)}\}.
\end{aligned}	\label{eq:ue}
\end{equation}

\subsection{Problem Formulation}
We aim to maximize the average flow passing through the highway. To
achieve this goal, we select our objective function to be
$\mathbb{E}_{\prob_\varpi}\{\frac{1}{T}
\sum_{e\in\edges,t\in\mathcal{T}}{f_e(t)}\}$, where $f_e(t):=\rho_e(t)
u_e$, for each $e\in\mathcal{E}$, $t\in\mathcal{T}$, and the notion
$\prob_{\varpi}$ is the distribution of the concatenated random
variable $\varpi:=(\rho(0),\omega)$.
Given the parameters $\{\overline{f}_e\}_{e\in\edges}$,
$\{\overline{\rho}_e\}_{e\in\edges}$,
$\{\supscr{f}{U}_e\}_{e\in\edges}$,
$\{\supscr{\rho}{U}_e\}_{e\in\edges}$ and $\Gamma$, the problem of
computing variable speed limits which are robust to the uncertainty
$\varpi$, can be formulated as follows: {\leqnomode
\begin{align}
  \label{eq:P} \tag{$\mathbf{P}$}~ \max\limits_{ \substack{u, \rho }}
  \quad & \mathbb{E}_{\prob_\varpi}\{ \frac{1}{T} \sum_{e\in\edges,t\in\mathcal{T}}{\rho_e(t) u_e}\} ,  \\
  \st \quad & \eqref{eq:RhoC},\; \eqref{eq:CTM},\; \eqref{eq:ue}, \;
  \nonumber
\end{align}}where
$\rho:=(\rho_1(1),\rho_2(1),\ldots,\rho_n(1),\rho_1(2),\ldots,\rho_n(T))$
is the concatenated variable of
$\{\rho_e(t)\}_{e\in\edges,t\in\mathcal{T}\setminus \{0 \} }$ and $u:=(u_1,\ldots,u_n)$
is that of $\{u_e\}_{e\in\edges}$.

The probability distribution $\prob_\varpi$ is needed in order to
compute a set of speed limits which are robust to the uncertainty
$\varpi$ and solve problem~\eqref{eq:P}. However, this distribution
$\prob_{\varpi}$ is unknown, and we assume that we have access to $N$
samples of the random variable $\varpi$. Thus, we investigate the
computation of a set of feasible variable speed limits that possess
certain out-of-sample guarantees within a distributionally robust
optimization framework~\cite{PME-DK:17,AC-JC:17}. In this way, we seek
to find a set of feasible ${u}$ with certificate ${J}({u})$, such that
the out-of-sample performance of $u$, $\mathbb{E}_{\prob_{\varpi}} \{
\frac{1}{T} \sum_{e\in\edges,t\in\mathcal{T}}{\rho_e(t) u_e}\}$, has
the following performance guarantee with a given confidence level
$\beta \in (0,1)$:
\begin{equation}
  {{P}^N}\left(\mathbb{E}_{\prob_\varpi} 
    \{ \frac{1}{T}  \sum_{e\in\edges,t\in\mathcal{T}}{\rho_e(t) u_e}\} \geq
    {J}({u}) \right)\geq 1- \beta,
\label{eq:perf}
\end{equation}
where ${P}^N$ denotes the probability that the event
$\mathbb{E}_{\prob_{\varpi}} \{ \frac{1}{T}
\sum_{e\in\edges,t\in\mathcal{T}}{\rho_e(t) u_e}\} \geq {J}({u})$
happens on the $N$ product of the sample space that defines $\varpi$.

\section{Performance Guaranteed Reformulation}\label{sec:Reform}
Problem~\eqref{eq:P} is intractable mainly due to the  uncertainty
$\varpi$. We aim to obtain a tractable reformulation of~\eqref{eq:P} that
enables us to compute the variable speed limits with performance
guarantees, as shown in~\eqref{eq:perf}. To achieve this goal, we
follow a four-step procedure.  First, we treat the density
trajectories as random variables and formulate Problem~\eqref{eq:P}
into an equivalent problem, Problem~\eqref{eq:P1}. Second, we
propagate the sample trajectories via the measurements of
$\varpi$. This step enables the the distributionally robust
optimization framework for dynamical systems with performance
guarantees be equivalent to~\eqref{eq:perf}. Third, we adapt the
distributionally robust optimization approach to Problem~\eqref{eq:P1}
for certificates. Finally, we obtain a tractable problem reformulation
for data-driven solutions and certificates.

\noindent \textbf{Step 1: (Equivalent Formulation of~\eqref{eq:P})}
  The random inflows and outflows along the highway
result in random density dynamics characterized
by~\eqref{eq:CTM}. Therefore, the densities $\rho_e(t)$, for all $e
\in \edges$ and $t \in \mathcal{T}\setminus\{ 0\}$, will be random
variables whose distributions are determined by speed limits $u$,
inflows and outflows $\omega$, and the initial density
$\vect{\rho}(0)$. In this step, we take the decision variable $\rho$
in Problem~\eqref{eq:P} as the random variable. Using Probability
Theory, we derive an equivalent Problem~\eqref{eq:P1} via a
reformulation of the constraints in~\eqref{eq:P}.

Let us take the variable $\rho$ 
considered in~\eqref{eq:P} as the random variable. For each speed
limit $u$ characterized by~\eqref{eq:ue}, let $\mathcal{Z}(u)$ and
$\prob(u)$ denote the support of $\rho$ and the probability
distribution of $\rho$, respectively. Recall that the support of
$\rho$ is the smallest closed set such that the probability $P(\rho
\in \mathcal{Z}(u))=1$. Note that in Problem~\eqref{eq:P},
constraints~\eqref{eq:RhoC} on $\rho$ and $u$ ensure no
congestion. Therefore, to obtain an equivalent problem, we need to
select $\mathcal{Z}(u) $ such that $\mathcal{Z}(u) \subseteq
\setdef{\rho \in \real^{nT}}{\eqref{eq:RhoC}}$. Without loss of
generality, we select $\mathcal{Z}(u):=\setdef{\rho \in
  \real^{nT}}{\eqref{eq:RhoC}}$. To fully characterize the random
variable $\rho$, we need to determine the distribution
$\prob(u)$. Using the density dynamics~\eqref{eq:CTM}, we can
represent $\prob(u)$ as a convolution of the distribution
$\prob_{\varpi}$. Given that $\prob_{\varpi}$ is unknown, the
characterization of $\prob(u)$ is done in later steps.

We denote by ${\mathcal{M}}(\mathcal{Z}(u))$ the space of all
probability distributions supported on $\mathcal{Z}(u)$, and
equivalently write the unsolvable Problem~\eqref{eq:P} as {\leqnomode
\begin{align}
  \label{eq:P1} \tag{$\mathbf{P1}$}~\max\limits_{u} \quad &
  \mathbb{E}_{\prob(u)}\{ H(u;{\rho}):= \frac{1}{T}
  \sum_{e\in\edges,t\in\mathcal{T}}{\rho_e(t) u_e}\} , \\
  \st \quad & \prob(u) \textrm{ characterized by~\eqref{eq:CTM} and } \prob_{\varpi}, \nonumber \\
  & \prob(u) \in {\mathcal{M}}(\mathcal{Z}(u)), \;
  \eqref{eq:ue}. \nonumber
\end{align}}We can now obtain the performance guarantee
of~\eqref{eq:P1} by considering the induced out-of-sample performance
on $\mathbb{P}(u)$, written as $\mathbb{E}_{\prob(u)} \{
\frac{1}{T}\sum_{e\in\edges,t\in\mathcal{T}}{\rho_e(t) u_e}\}$. For
all Problems derived later, we will use the performance guarantees
equivalent to~\eqref{eq:perf}, as follows:
\begin{equation}
  {{P}^N}\left(\mathbb{E}_{\prob(u)} 
    \{ \frac{1}{T}  \sum_{e\in\edges,t\in\mathcal{T}}{\rho_e(t) u_e}\} \geq
    {J}({u}) \right)\geq 1- \beta,
\label{eq:perfnew}
\end{equation} 
where notions ${P}^N$, ${J}({u})$ and $\beta$ in~\eqref{eq:perfnew}
are those as in~\eqref{eq:perf}.

\noindent \textbf{Step 2: (Sample Trajectory Propagators)} In this
step, we obtain samples of $\rho$ and use them to deal with
$\prob(u)$. Consequently, these samples enable the distributionally
robust optimization framework for~\eqref{eq:P1}.

Given the speed limit $u \in \Gamma$, the density dynamics represented
by~\eqref{eq:CTM} reduce to a linear system. As the result of the
Uniqueness Solutions of Linear Systems, we can use~\eqref{eq:CTM} to
obtain a unique density trajectory $\rho$ for each measurement of
$\varpi$. As mentioned earlier, we assume that a set of data
comprising $N$ samples of random variable $\varpi$ is available. Let
$\mathcal{L}= \untilone{N}$ denote the index set for realizations of
the random variable ${\varpi}$, and let us denote the set of
independent and identically distributed (iid) realizations of $\varpi$
by $\{\varpi^{(l)}:=(\rho^{(l)}(0),
\omega^{(l)})\}_{l\in\mathcal{L}}$.  Given these realizations
$\{\varpi^{(l)}\}_{l\in\mathcal{L}}$, the sample trajectories
$\{{\rho}^{(l)}\}_{l\in\mathcal{L}}$ of the random traffic flow
dynamics for each edge $e \in \edges$ with its precedent edge $s\in
\edges\cup \varnothing$, are given by
\begin{equation}
  {\rho}^{(l)}_e(t+1)= {\rho}^{(l)}_e(t)+ 
  h(u_s {\rho}^{(l)}_s(t)- u_e {\rho}^{(l)}_e(t)+ \omega^{(l)}_e(t)),
	\label{eq:Rhohat}\tag{2a}
\end{equation}
for all $t \in \mathcal{T}\setminus \{ T\}$ and sample index $l \in
\mathcal{L}$. The following lemma establishes that
$\{{\rho}^{(l)}\}_{l\in\mathcal{L}}$ are iid samples from $\prob(u)$.
\begin{lemma}[Iid sample generators of ${\rho}$] Given $u \in \Gamma$
  and iid realizations $\{\varpi^{(l)}\}_{l\in\mathcal{L}}$ of
  $\varpi$, the system dynamics~\eqref{eq:Rhohat} generates iid sample
  trajectories $\{{\rho}^{(l)}\}_{l\in\mathcal{L}}$ of $\prob(u)$.
\label{lemma:RhoGener}
\end{lemma}
\begin{proof}
We know that continuous functions of iid random variables are
iid, therefore the sample trajectories
$\{{\rho}^{(l)}\}_{l\in\mathcal{L}}$ generated
by~\eqref{eq:Rhohat} are iid realizations of $\prob(u)$. 
\end{proof}

Consider the random variable $\varpi$ with unknown distribution
$\prob_{\varpi}$. Let $\subscr{\mathcal{M}}{lt}(\mathcal{Z}_{\varpi})
\subset {\mathcal{M}}(\mathcal{Z}_{\varpi})$ denote the space of all
light-tailed probability distributions supported on
$\mathcal{Z}_{\varpi}$. We make the following assumption on
$\prob_{\varpi}$:
\begin{assumption}[Light tailed unknown distributions]
  It holds that $\prob_{\varpi} \in
  {\subscr{\mathcal{M}}{lt}}(\mathcal{Z}_{\varpi})$, i.e., there
  exists an exponent $a>1$ such that: $b:= \mathbb{E}_{\prob_{\varpi}}
  [\exp(\Norm{\varpi}^a)] < \infty$. \label{assump:1}
\end{assumption}

The above assumption invokes the following lemma:
\begin{lemma}[Light-tailed distribution of ${\rho}$] If
  Assumption~\ref{assump:1} holds, then $\prob(u) \in
  \subscr{\mathcal{M}}{lt}(\mathcal{Z}{(u)})$. 
\label{lemma:hatRho}
\end{lemma}
\begin{proof}
  To show the random variable ${\rho}$ has a light-tailed
  distribution, we bound its norm by $\Norm{\varpi}$ via an norm
  equivalence and dynamics on ${\rho}$.  In this way, by norm
  equivalence, there exists $M_1>0$ such that $\Norm{{\rho}} \leq M_1
  \Norm{{\rho}}_{\infty}$. Let $t^{\star} \in \argmax_{t \in
    \mathcal{T}\setminus \{ 0\}} \{ \Norm{\rho(t)}_{\infty} \}$, we
  have:
	\begin{equation*}
          \begin{aligned}
            \Norm{{\rho}} &\leq M_1 \Norm{{\rho}}_{\infty}, \\
            &=M_1 \max_{t \in \mathcal{T}\setminus \{ 0\}} \{ \Norm{\rho(t)}_{\infty} \}=M_1 \Norm{\rho(t^{\star})}_{\infty}. \\
          \end{aligned}
\end{equation*}				
Note that $\rho(t^{\star})$ is the density at time slot
$t^{\star}$. Given that $\rho(t^{\star})$ is the function of $\omega$,
$\rho(0)$ and $u$, we let $A(u)$ denote the matrix that is consistent
with the graph $\mathcal{G}$, and write the density $\rho(t^{\star})$
in the following compact form:
\begin{equation*}
  \rho(t^{\star})= 
  (I_n + hA(u))^{t^{\star}}\rho(0)+ h \sum\limits_{\tau=0}^{t^{\star}-1} \omega(\tau),
	\label{eq:Rhohat2}
\end{equation*}
with
\[A(u)=\begin{bmatrix}
    -u_{1} &       &        &        &   \\
    u_{1} & -u_{2} &        &        &  \\
          &  u_{2} & \ddots &        &  \\		
          &        & \ddots & \ddots &  \\
	        %&        &        & \ddots &  \ddots &\\
          &        &        & u_{n-1}& -u_{n}
\end{bmatrix}.
\]		
	Let $M_2:= \max\{M_1 \Norm{(I_n +
    hA)^{t^{\star}}}, \; M_1h \}$, we compute:
	\begin{equation*}
	\begin{aligned}
          \Norm{{\rho}} &\leq M_1 \Norm{\rho(t^{\star})}_{\infty}, \\
	&=M_1 \Norm{ (I_n + hA)^{t^{\star}}\rho(0) +
            h \sum\limits_{\tau=0}^{t^{\star}-1} \omega(\tau) }_{\infty}, \\
          &\leq M_1 \Norm{(I_n + hA)^{t^{\star}}}
          \Norm{\rho(0)}_{\infty} +
          M_1h\sum\limits_{\tau=0}^{t^{\star}-1} \Norm{\omega(\tau)}_{\infty}, \\
          & \leq M_2\left(\Norm{\rho(0)}_{\infty}
            +\sum\limits_{\tau=0}^{t^{\star}-1}
            \Norm{\omega(\tau)}_{\infty} \right), \\
						&\leq M_2(t^{\star}+1)\Norm{\varpi}_{\infty}.
		\end{aligned}
\end{equation*}	
Again using norm equivalence, there exists $M_3>0$ such
that $\Norm{\varpi}_{\infty} \leq M_3 \Norm{\varpi}$. This results in
\begin{equation*}
	\begin{aligned}
          \Norm{{\rho}} &\leq M_2(t^{\star}+1)\Norm{\varpi}_{\infty} \leq M_2M_3(t^{\star}+1)\Norm{\varpi}. \\
	\end{aligned}
\end{equation*}
Let $M_4=({M_2M_3(t^{\star}+1)})^a < \infty$. Then for each $u$ and
any $a>1$ such that $\mathbb{E}_{\prob_{\varpi}}
[\exp(\Norm{\varpi}^a)] < \infty$, we have $\mathbb{E}_{\prob(u)}
[\exp(\Norm{{\rho}}^a)]= \mathbb{E}_{\prob_{\varpi}}
[\exp(\Norm{{\rho}(u,\varpi)}^a)]\leq \mathbb{E}_{\prob_{\varpi}}
[\exp( M_4 \Norm{\varpi}^a)]< \infty$, that is, $\prob(u)$ is light
tailed.
\end{proof}

The above lemma is the last ingredient to enable the distributionally
robust optimization framework for~\eqref{eq:P1} in the next step.

\noindent \textbf{Step 3: (Certificates)} We now design a certificate
to satisfy the performance guarantee~\eqref{eq:perfnew} using the
distributionally robust optimization approach. 
To design a certificate
${J}({u})$ for a given set of speed limits ${u}$, we need to estimate
the probability distribution ${\prob}(u)$ empirically. To do so, we
use the sample trajectories $\{{\rho}^{(l)}\}_{l\in\mathcal{L}}$
obtained from sample generators~\eqref{eq:Rhohat}.  Let
$\hat{\prob}(u):= ({1}/{N})\sum_{l\in\mathcal{L}}
\delta_{\{{\rho}^{(l)}\}}$ denote the estimated probability
distribution. In this way, by application of the point mass operator
$\delta$, we have ${\mathbb{E}_{\mathbb{\hat{P}}(u)}} \{
H(u;{\rho})\}= ({1}/{N})\sum_{l\in\mathcal{L}} H({u};{\rho}^{(l)})$,
which is taken to be the candidate certificate for the performance
guarantee~\eqref{eq:perfnew}.

Note that such certificates only result in an approximation of the
out-of-sample performance if $\prob$ is unknown, and~\eqref{eq:perfnew}
cannot be guaranteed in probability. To achieve the out-of-sample
performance, we follow the procedure proposed
in~\cite{AC-JC:17,DL-SM:18-extended}. More precisely, we determine an
\textit{ambiguity set} $\mathcal{\hat{P}}({u})$ containing all the
possible probability distributions supported on $\mathcal{Z}({u})
\subseteq \real^{nT}$ that can generate the sample trajectories $\{{\rho}^{(l)}\}_{l\in\mathcal{L}}$ with
high confidence. Then, with the given data-driven solution ${u}$, it
is plausible to consider the worst-case expectation of the
out-of-sample performance for all distributions contained in
$\mathcal{\hat{P}}({u})$. Such worst-case distribution offers a lower
bound for the out-of-sample performance with high probability.

Lemma~\ref{lemma:hatRho} on the light-tailed distribution of
${\rho}$ validates the modern measure concentration
result~{\cite[Theorem~2]{NF-AG:15}} on
$\subscr{\mathcal{M}}{lt}(\mathcal{Z}({u}))$, which provides an
intuition for considering the Wasserstein ball
$\mathbb{B}_{\epsilon}(\hat{\prob}({u}))$ of center $\hat{\prob}({u})$
and radius $\epsilon$ as the ambiguity set $\mathcal{\hat{P}}({u})$.
The ambiguity sets
$\mathcal{\hat{P}}({u}):=\mathbb{B}_{\epsilon}(\hat{\prob}({u}))$
allow us to provide the certificate that ensures the performance
guarantee in~\eqref{eq:perfnew} for any data-driven solution ${u}$, by
taking ${J}({u}):= \inf_{\mathbb{Q} \in \mathcal{\hat{P}}({u})}
{\mathbb{E}_{\mathbb{Q}} \{H(u;{\rho}) \} }$.

\noindent \textbf{Step 4: (Tractable Reformulation of~\eqref{eq:P1})}
To obtain the certificate ${J}({u})$, we need to solve an
infinite-dimensional optimization problem, which is generally hard.
With an extended version of the strong duality results for moment
problem~\cite[Lemma~3.4]{AS-MG-MAL:01}, we can reformulate the optimization
problem for ${J}({u})$ into a finite-dimensional convex programming
problem as the following: 
\begin{equation*}\small
	\begin{aligned}
          J(u)=\sup_{\lambda \geq 0} \; &- \lambda\epsilon(\beta) +
          \frac{1}{N} \sum_{l\in\mathcal{L}} \inf_{{\rho} \in
            \mathcal{Z}(u) } \left(\lambda
            \Norm{{\rho}-{\rho}^{(l)}} +{H(u;{\rho})}
          \right), \\
					 \st \; & \{{\rho}^{(l)}\}_{l\in\mathcal{L}} \textrm{ is obtained from }\eqref{eq:Rhohat},
	\end{aligned}
\end{equation*}
\normalsize where the parameter $\beta$ is the confidence level
in~\eqref{eq:perfnew} and the value $\epsilon(\beta)$ is the radius of
$\mathbb{B}_{\epsilon(\beta)}$ as calculated
in~\cite{DL-SM:18-extended}.  To obtain a data-driven speed limits $u$
with a good out-of-sample performance of~\eqref{eq:P1}, we need to
obtain $u$ with a high certificate ${J}({u})$. Finally, we can obtain
a data-driven solution ${u}$ with a high certificate ${J}({u})$, by
solving the problem: {\leqnomode
\begin{align}
  \label{eq:P2} \tag{$\mathbf{P2}$}~\max\limits_{u, \; \st \;
    {\eqref{eq:ue}}} \quad & {J}({u}).
\end{align}}

Problem~\eqref{eq:P2} consists of many inner optimization problems. To
propose a solution method, we consider an equivalent optimization
problem given as follows:

{\leqnomode
\begin{align}
  \label{eq:P3} \tag{$\mathbf{P3}$}~
  \max\limits_{u,{\rho},\lambda,\mu,\nu,\eta } \quad & -\lambda
  \epsilon(\beta) - 
  \frac{1}{N}\sum_{e,t,l}{\overline{f}_e\overline{\rho}_e\eta_e^{(l)}(t)  } \\
  & \hspace{2.5cm} +\frac{1}{N}\sum_{l}{\left\langle \nu^{(l)} , {\rho}^{(l)}\right\rangle}, \nonumber \\
  \st \quad & [\overline{f}+u\circ(\overline{\rho}-
  \rho^{c}(\overline{u}))]\otimes \vectorones{T} \circ \eta^{(l)}  \nonumber \\
  & \hspace{2.5cm} -\mu^{(l)} \geq \vectorzeros{nT},  
  \; \forall \; l\in\mathcal{L}, \nonumber \\
  & \nu^{(l)} =\mu^{(l)}+ \frac{1}{T}u \otimes \vectorones{T},
  \; \forall \; l\in\mathcal{L}, \nonumber \\
  & \Norm{\nu^{(l)}}_{\star} \leq \lambda, \; \forall 
  \; l\in\mathcal{L}, \nonumber\\
  & \eta^{(l)} \geq \vectorzeros{nT}, \; \forall \; l\in\mathcal{L},\nonumber \\
  & \eqref{eq:ue},\; \eqref{eq:Rhohat},\nonumber
\end{align}}where decision variables $(u,{\rho},\lambda,\mu,\nu,\eta)$
are concatenated versions of $ u_e$, ${\rho}^{(l)}_e(t)$, $\lambda$,
$\mu^{(l)}_e(t)$, $\nu^{(l)}_e(t)$, $\eta^{(l)}_e(t) \in \real$, for
all $l \in\mathcal{L}$, $t \in \mathcal{T} \setminus \{ 0\}$, and $e
\in \edges$. The parameter $\epsilon(\beta)$ is the radius of
$\mathbb{B}_{\epsilon(\beta)}$, 
the value $\rho^{c}(\overline{u}):=\overline{f}/\overline{u}$ is the critical
density under the free flow and $\overline{\rho}$ is the jam density.

The following lemma shows that problems~\eqref{eq:P2}
and~\eqref{eq:P3} are equivalent for $({u},{J})$.

\begin{lemma}[Tractable reformulation of~\eqref{eq:P2}]
  Consider the DRO setting as
  in~\eqref{eq:P2}. Then Problem~\eqref{eq:P2} is equivalent
  to~\eqref{eq:P3} in the sense that their optimal objective value are
  the same and the set of optimizers of~\eqref{eq:P2} is the
  projection of that of~\eqref{eq:P3}.
  Further, for any feasible point $(u,{\rho},\lambda,\mu,\nu,\eta)$
  of~\eqref{eq:P3}, let $\hat{J}(u)$ denote the value of its objective function. Then the pair $({u},\hat{J}(u))$ gives a data-driven solution $u$ with an estimate
  of its certificate ${J}(u)$ by $\hat{J}(u)$, such that the performance
  guarantee~\eqref{eq:perfnew} holds for $(u,\;\hat{J}(u) )$. 
\label{lemma:P2}
\end{lemma}
\begin{proof}
  Following~\cite{AC-JC:17,DL-SM:18-extended}, and references therein,
  we write Problem~\eqref{eq:P2} as follows:
\begin{equation*}
  \begin{aligned}
    \sup_{u,\lambda \geq 0} \;& - \lambda\epsilon(\beta) + \frac{1}{N}
    \sum_{l\in\mathcal{L}} \inf_{{\rho} \in \mathcal{Z}(u) }
    \left(\lambda \Norm{{\rho}-{\rho}^{(l)}}+
      {H(u;{\rho})} \right) ,\\
    \st &  \; \eqref{eq:ue},\; \eqref{eq:Rhohat}. \\
	\end{aligned}
\end{equation*}
Using the definition of the dual norm and moving its $\sup$ operator
we can write the above problem as:
\begin{equation*}
	\begin{aligned}
\sup_{u,\lambda \geq 0} \;& - \lambda\epsilon(\beta) + \frac{1}{N} \sum_{l\in\mathcal{L}} \inf_{{\rho} \in \mathcal{Z}(u) } \sup_{\Norm{\nu^{(l)}}_{\star}\leq \lambda}   \\
& \hspace{2cm} \left({\left\langle \nu^{(l)} , {\rho}-{\rho}^{(l)}\right\rangle}  +{H(u;{\rho})} \right),\\
\st &  \; \eqref{eq:ue},\; \eqref{eq:Rhohat}. \\
	\end{aligned}
\end{equation*}
Given $\lambda \geq 0$, the sets $\setdef{\nu^{(l)} \in
  \real^{nT}}{{\Norm{\nu^{(l)}}_{\star}\leq \lambda}}$ are compact for
all $l\in \mathcal{L}$. We then apply the minmax theorem between
$\inf$ and the second $\sup$ operators. This results in the switch of
the operators, and by combining the two $\sup$ operators we have:
\begin{equation*}
	\begin{aligned} 
          \sup_{u,\lambda, \nu} & - \lambda\epsilon(\beta) + \frac{1}{N}
          \sum_{l\in\mathcal{L}} \inf_{{\rho} \in \mathcal{Z}(u) }
          \left({\left\langle \nu^{(l)} , {\rho}-{\rho}^{(l)}
              \right\rangle} +{H(u;{\rho})} \right),\\
          \st &  \; \eqref{eq:ue},\; \eqref{eq:Rhohat},\; \lambda \geq 0, \\
          & \Norm{\nu^{(l)}}_{\star}\leq \lambda, \; \forall l \in
          \mathcal{L}. 
	\end{aligned}
\end{equation*}
The objective function can be simplified as follows: 
\begin{equation*}
	\begin{aligned}
          - \lambda\epsilon(\beta) +
          \frac{1}{N}\sum_{l\in\mathcal{L}}{\left\langle -\nu^{(l)} ,
              {\rho}^{(l)}\right\rangle} + \frac{1}{N}
          \sum_{l\in\mathcal{L}} h^{(l)}(u),\\
	\end{aligned}
\end{equation*}
where 
\begin{equation*}
	\begin{aligned}
          h^{(l)}(u):=&\inf_{{\rho} \in \mathcal{Z}(u) }
          \left({\left\langle \nu^{(l)} , {\rho}\right\rangle}
            +{H(u;{\rho})} \right), \; \forall l\in \mathcal{L}.
	\end{aligned}
\end{equation*}
For each $l\in \mathcal{L}$, we rewrite $h^{(l)}(u)$ by firstly taking a minus sign out of the $\inf$ operator, then exploiting the equivalent representation of $\sup$ operation, and finally using the definition of conjugate functions. The function $h^{(l)}(u)$ results in the following form:
	\begin{equation*}
	\begin{aligned}					
h^{(l)}(u)=&-\sup_{{\rho} \in \mathcal{Z}(u) }
          \left({\left\langle -\nu^{(l)} , {\rho}\right\rangle}
            -{H(u;{\rho})} \right), \\
          =& -\sup_{{\rho} } \left({\left\langle -\nu^{(l)} ,
                {\rho}\right\rangle} -{H(u;{\rho})} -
            \chi_{\mathcal{Z}(u)}({\rho}) \right), \\
          =& - \left[ {H(u;\cdot)} + \chi_{\mathcal{Z}(u)}(\cdot)
          \right]^{\star}(-\nu^{(l)}).
	\end{aligned}
\end{equation*}		
Further, we apply the property of the inf-convolution operation and push the minus sign back into the $\inf$ operator, for each $h^{(l)}(u)$, $l\in \mathcal{L}$. The representation of $h^{(l)}(u)$ results in the following relation:
		\begin{equation*}
	\begin{aligned}					
h^{(l)}(u)=& - \inf_{\mu} \left( \left[ {H(u;\cdot)}\right]^{\star}
            (-\mu^{(l)}-\nu^{(l)}) \right. \\
          & \hspace{2cm} \left. +\left[ \chi_{\mathcal{Z}(u)}(\cdot)
            \right]^{\star}(\mu^{(l)}) \right), \\
          =& \sup_{\mu} \left( -\left[ {H(u;\cdot)}\right]^{\star}
            (-\mu^{(l)}-\nu^{(l)}) \right. \\
          & \hspace{2cm} \left.-\left[
              \chi_{\mathcal{Z}(u)}(\cdot)\right]^{\star}
            (\mu^{(l)}) \right). \\
	\end{aligned}
\end{equation*}
By substituting $-\nu^{(l)}$ by $\nu^{(l)}$, the resulting
optimization problem has the following form:
\begin{equation*}
	\begin{aligned}
          \sup_{u,\lambda,\mu, \nu} \; & - \lambda\epsilon(\beta)
          -\frac{1}{N} \sum_{l\in\mathcal{L}} \left( \left[
              {H(u;\cdot)}\right]^{\star}
            (-\mu^{(l)}+\nu^{(l)}) \right. \\
          & \hspace{1cm} \left. +\left[ \chi_{\mathcal{Z}(u)}(\cdot)
            \right]^{\star}(\mu^{(l)}) -{\left\langle \nu^{(l)} ,
                {\rho}^{(l)}\right\rangle} \right), \\
          \st &  \; \eqref{eq:ue},\; \eqref{eq:Rhohat}, \; \lambda \geq 0,  \\
          & \Norm{\nu^{(l)}}_{\star}\leq \lambda, \; \forall
          l \in \mathcal{L}. 
	\end{aligned}
\end{equation*}
Given $u$, the strong duality of linear programs are applicable for
the conjugate of the function ${H(u;\cdot)}$ and the support function
$\sigma_{\mathcal{Z}(u)}(\mu^{(l)})$. Using the strong duality and the
definition of the support function, we compute
\begin{equation*}
  \begin{aligned}
          \left[ {H(u;\cdot)}\right]^{\star}
          &(\nu^{(l)}-\mu^{(l)})\\
          &:=\begin{cases}
            0, & \nu^{(l)}=\mu^{(l)}+ \frac{1}{T}u \otimes \vectorones{T},\; \forall \; l\in\mathcal{L}, \\
            \infty, & {\rm{o.w.}}, \\
          \end{cases}
	\end{aligned}
      \end{equation*}
      and
      \begin{equation*}
        \begin{aligned}
          &\left[ \chi_{\mathcal{Z}(u)}(\cdot)\right]^{\star}(\mu^{(l)})=
          \sigma_{\mathcal{Z}(u)}(\mu^{(l)}) \\
          &= \begin{cases}
            \sup\limits_{\xi} & {\left\langle \mu^{(l)} , 
                \xi \right\rangle}, \\
            \st &\;  0 \leq \xi_e(t) \leq \rho^{c}_{e}(u_e),~\forall e\in\edges, \\
          \end{cases}  \\
          &= \begin{cases}
            \inf\limits_{\eta} & 
            \sum\limits_{e\in\mathcal{E},t\in\mathcal{T},l\in\mathcal{L}}{
              \overline{f}_e\overline{\rho}_e\eta_e^{(l)}(t)  }, \\
            \st & [\overline{f}+u
            \circ(\overline{\rho}-\rho^{c}(\overline{u}))]\otimes 
            \vectorones{T} \circ \eta^{(l)} \\
            & \hspace{3.5cm} -\mu^{(l)} \geq \vectorzeros{nT},  
            \; \forall \; l\in\mathcal{L}, \\
            &  \eta^{(l)} \geq \vectorzeros{nT},\; \forall \; l\in\mathcal{L}. \\
          \end{cases} 
	\end{aligned}
\end{equation*}

By substituting these functions into the objective function and take a
minus sign out of the resulting $\inf$ operator above, we obtain the
form of Problem~\eqref{eq:P3}. Given that the relations hold with
equalities, we therefore claim that~\eqref{eq:P2} is equivalent
to~\eqref{eq:P3}.

Further, given any feasible point $(u,{\rho},\lambda,\mu,\nu,\eta)$
of~\eqref{eq:P3}, we denote its objective value by $\hat{J}(u)$.  The
value $\hat{J}(u)$ is a lower bound of~\eqref{eq:P3} and therefore a
lower bound for~\eqref{eq:P2}, i.e., $\hat{J}(u) \leq {J}(u)$. Then
the value $\hat{J}(u)$ is an estimate of the certificate for the
performance guarantee~\eqref{eq:perfnew}. Therefore, $(u,\;\hat{J}(u)
)$ is a data-driven solution and certificate pair for~\eqref{eq:P1}.
\end{proof}

Problem~\eqref{eq:P3} is inherently difficult to solve due to the
discrete decision variables $u$, bi-linear terms in the first group of
constraints $u \otimes \vectorones{T} \circ \eta^{(l)}$, and the
nonlinear sample trajectories
$\{{\rho}^{(l)}\}_{l\in\mathcal{L}}$, which motivates our next section.

\section{Solution Techniques for Nonconvex
  Problem~\eqref{eq:P3}} \label{sec:Solution}
To compute high-quality solutions, we follow a two-step procedure. In
the first step, we transform Problem~\eqref{eq:P3} into a
mixed-integer bi-linear program with a linear constrained set. We call
it Problem~\eqref{eq:P4}. Finally, we propose an integer-solution
search algorithm to compute high-quality solutions to
Problem~\eqref{eq:P4}.

\noindent \textbf{Step 1:} In this step, we represent the speed limits
$u$ with a set of binary variables, and then represent each bi-linear
term that is comprised of a continuous variable and a binary variable,
with a set of linear constraints.

\textit{Binary Representation of Speed Limit $u$:} For each edge
$e\in\mathcal{E}$ and speed limit value $\gamma^{(i)} \in \Gamma$ with
$i \in \mathcal{O}:=\untilone{m}$, let us define the binary variable
$x_{e,i}$ to be equal to one if $u_e=\gamma^{(i)}$; otherwise
$x_{e,i}=0$. We will then have $u_e=\sum_{i\in \mathcal{O}}
\gamma^{(i)} x_{e,i}$ for each $e \in\mathcal{E}$. Using this
representation, we reformulate the speed limit
constraints~\eqref{eq:ue} into the following:
\begin{equation}
	\begin{aligned}
          \supscr{u}{L}_e \leq & \sum_{i \in \mathcal{O}} \gamma^{(i)}
          x_{e,i} \leq \supscr{u}{U}_e, \; \sum_{i \in
            \mathcal{O}}x_{e,i}=1,
          \; \forall e \in \edges, \\
          &x_{e,i} \in \{0,\; 1\}, \; \forall e \in \edges, \;
          i \in \mathcal{O}, \\
	\end{aligned} 	\label{eq:ueNew}
\end{equation}
and we update the sample trajectories formula~\eqref{eq:Rhohat} as
follows:
\begin{equation}
\begin{aligned}
  {\rho}^{(l)}_e(t+1)=&
  {\rho}^{(l)}_e(t) +h \omega^{(l)}_e(t) \\
  &+ h\sum_{i\in \mathcal{O}} \gamma^{(i)} (x_{s,i}
  {\rho}^{(l)}_s(t)- x_{e,i} {\rho}^{(l)}_e(t)), \\
  & \hspace{1.2cm} \;\forall e\in\edges,\;t\in\mathcal{T}
  \setminus\{T\},\;l\in\mathcal{L}, \\
\end{aligned}
	\label{eq:RhohatNew}
\end{equation}

\textit{Reformulation of Bi-linear Terms:} In Problem~\eqref{eq:P3},
there are three groups of bi-linear terms: 1) the bi-linear terms
$\nu^{(l)}_e(t){\rho}^{(l)}_e(t)$ in the objective function
written as $\nu^{(l)} \circ {\rho}^{(l)}$, 2) the bi-linear terms
$\sum_{i\in \mathcal{O}} \gamma^{(i)} x_{e,i}\eta^{(l)}_e(t)$ which
appear in the first set of constraints (e.g., $u \otimes
\vectorones{T} \circ \eta^{(l)}$), and 3) the bi-linear terms $x_{e,i}
{\rho}^{(l)}_e(t)$ in the sample trajectories
formula~\eqref{eq:RhohatNew}. In the group 2) and 3), each bi-linear
term is comprised of a continuous variable and a binary variable. In
this regard, we simplify these bi-linear terms by using the
linearization technique under the following assumption:

\begin{assumption}[Bounded dual variable $\eta$]
  There exist large enough scalar $\overline{\eta}$ such that
  $\eta^{(l)}_e(t) \leq \overline{\eta}$ for all $e \in \edges$, $t
  \in \mathcal{T} \setminus \{ 0\}$ and $l
  \in\mathcal{L}$. \label{assump:2}
\end{assumption}

\begin{proposition}[Equivalence reformulation for bi-linear terms in
  group 2) and 3)~{\cite[Section 2]{FG:75}}]\label{glov} Let
  $\mathcal{Y}\subset\mathbb{R}$ be a compact set. Given a binary
  variable $x$ and a linear function $g(y)$ in a continuous variable
  $y\in \mathcal{Y}$, $z$ equals the quadratic function $xg(y)$ if and
  only if
\begin{align}
  &\underline{g}x \le z \le \overline{g}x,\nonumber\\
  &g(y)-\overline{g}\cdot(1-x) \le z\le g(y)-\underline{g}\cdot(1-x),\nonumber
\end{align}
where $\underline{g}=\min_{y\in \mathcal{Y}}\{g(y)\}$ and
$\overline{g}=\max_{y\in \mathcal{Y}}\{g(y)\}$.  \hfill $\square$
\end{proposition}
Applying Proposition~\ref{glov}, we introduce variables
$z^{(l)}_{e,i}(t)$ to represent the bi-linear terms
$x_{e,i}\eta^{(l)}_e(t)$ in the group 2), via the following
constraints:
\begin{equation}
\begin{array}{l}
  \sum_{i\in \mathcal{O}}z^{(l)}_{e,i}(t)=\eta^{(l)}_e(t), \; \\
  \hspace{2cm}\;  
  \forall e\in\edges, \; t\in\mathcal{T}\setminus\{0\},\; 
  l\in\mathcal{L}, \\
  0\leq z^{(l)}_{e,i}(t) \leq \overline{\eta} x_{e,i}, \\
  \hspace{2cm}\; \forall e\in\edges,\; i\in\mathcal{O},\;
  t\in\mathcal{T}\setminus\{0\},\;l\in\mathcal{L}, \\
  \eta^{(l)}_e(t) -\overline{\eta}(1- x_{e,i}) \leq z^{(l)}_{e,i}(t) 
  \leq \eta^{(l)}_e(t), \\
  \hspace{2cm}\; \forall e\in\edges,\; i\in\mathcal{O},\;t
  \in\mathcal{T}\setminus\{0\},\;l\in\mathcal{L}. \\
\label{eq:z}							
\end{array}
\end{equation}
Similarly, we introduce variables $y^{(l)}_{e,i}(t)$ to represent the
bi-linear terms $x_{e,i}{\rho}^{(l)}_e(t)$ in the group 3), via
the following constraints:
\begin{equation}
	\begin{array}{l}
          y^{(l)}_{e,i}(0)=x_{e,i}{\rho}^{(l)}_e(0), \\
          \hspace{2cm} \; 
          \forall e\in\edges,\; l\in\mathcal{L}, \; 
          i\in\mathcal{O},\\		
          {\rho}^{(l)}_e(t) -
          \overline{\rho}_e(1- x_{e,i}) \leq y^{(l)}_{e,i}(t) 
          \leq \overline{\rho}_e x_{e,i}, \\
          \hspace{2cm} \; \forall e\in\edges,\; i
          \in\mathcal{O},\;t\in\mathcal{T}\setminus\{0\},\;l
          \in\mathcal{L}, \\
	\end{array} \label{eq:G}
\end{equation}
\begin{equation}
\begin{array}{ll}
  \sum_{i\in \mathcal{O}}y^{(l)}_{e,i}(t)={\rho}^{(l)}_e(t), \\
  \hspace{2cm}\; \forall e\in\edges, \; 
  t\in\mathcal{T}\setminus\{0\},\; l\in\mathcal{L}, \\
  0\leq y^{(l)}_{e,i}(t) \leq {\rho}^{(l)}_e(t), \\
  \hspace{2cm}\; \forall e\in\edges,\; i\in\mathcal{O},
  \;t\in\mathcal{T}\setminus\{0\},\;l\in\mathcal{L}. \\
 \end{array}
\label{eq:y}
\end{equation}
Using variables $y^{(l)}_{e,i}(t)$ to reformulate the sample
trajectories formula~\eqref{eq:RhohatNew}, we have the following
constraints:
\begin{equation}
\begin{aligned}
  {\rho}^{(l)}_e(t+1)=
  &{\rho}^{(l)}_e(t)+h \omega^{(l)}_e(t)  \\
  &\; +h\sum_{i\in \mathcal{O}} \gamma^{(i)} (y^{(l)}_{s,i}(t)- y^{(l)}_{e,i}(t)),\\
  & \hspace{1cm} \;\forall e\in\edges,\;t\in\mathcal{T}
  \setminus\{T\},\;l\in\mathcal{L}, \\
\end{aligned} \label{eq:RhohatNew2}
\end{equation}
By the above reformulation, the bi-linear terms in group 2) and 3)
will be linear, and Problem~\eqref{eq:P3} can be equivalently written
as the following optimization problem:
{\leqnomode
\begin{align} \small
  \label{eq:P4}\tag{$\mathbf{P4}$}~ \max\limits_{\substack{x,y,z,
      {\rho}, \\ \lambda,\mu,\nu,\eta }} -\lambda \epsilon(\beta)
  - \frac{1}{N}\sum_{e,t,l}{\overline{f}_e\overline{\rho}_e\eta_e^{(l)}(t)
    +\nu^{(l)}_e(t){\rho}^{(l)}_e(t)},
\end{align}}  
{\reqnomode
\begin{align}
  \st \quad & \sum_{i\in \mathcal{O}} \gamma^{(i)}
  (\overline{\rho}-\rho^{c}(\overline{u}))\otimes
  \vectorones{T} \circ z^{(l)}_i  -\mu^{(l)} \nonumber \\
  & \hspace{1.8cm} +\overline{f}\otimes \vectorones{T} \circ
  \eta^{(l)} \geq \vectorzeros{nT}, \; \forall \;
  l\in\mathcal{L}, \label{eq:dual1} \\
  & \nu^{(l)} =\mu^{(l)}+ \frac{1}{T}\sum_{i\in \mathcal{O}} \gamma^{(i)}x_{i}
  \otimes \vectorones{T},\; \forall \;
  l\in\mathcal{L}, \label{eq:dual2} \\
  & \Norm{\nu^{(l)}}_{\star} \leq \lambda, \; \forall
  \; l\in\mathcal{L}, \label{eq:dual3} \\
  & \vectorzeros{nT} \leq \eta^{(l)} \leq \overline{\eta},
  \; \forall \; l\in\mathcal{L},\label{eq:dual4} \\
  & { \textbf{\small speed limits}}\; \eqref{eq:ueNew},
  \;{\textbf{\small dual variable}} \; \eqref{eq:z}, \nonumber \\
  & {\textbf{\small sample trajectories}} \{\eqref{eq:G}, \;
  \eqref{eq:y}, \; \eqref{eq:RhohatNew2} \}. \nonumber
\end{align}} 
Further, let
$\hat{J}(u)$ denote the value of the objective function
of~\eqref{eq:P4} at a computed feasible solution
$(x,y,z,{\rho},\lambda,\mu,\nu,\eta)$. Then, the resulting speed
limits $u:={\sum_{i\in \mathcal{O}} \gamma^{(i)}
  (x_{1,i},\ldots,x_{n,i})}$ provide a data-driven solution such that
$(u,\;\hat{J}(u))$ satisfies the performance
guarantee~\eqref{eq:perfnew}.

\noindent \textbf{Step 2:} Problem~\eqref{eq:P4} is computationally
intractable since its objective function is still nonlinear in its
arguments due to the bi-linear terms $\{\nu^{(l)} \circ
{\rho}^{(l)}\}_{l \in\mathcal{L}}$. To compute high-quality feasible
solutions to \eqref{eq:P4}, we propose an integer-solution search
algorithm. The proposed algorithm is a prototype of the
decomposition-based methods in the
literature~\cite{LX-TA-BP:11,LD-LX:16}. These methods can handle
specialized mix-integer nonlinear programs and achieve suboptimal
solutions efficiently.

We propose an integer-solution search algorithm as shown in
Algorithm~\ref{Alg:issa}. The idea of the algorithm is to iteratively
solve 1) upper-bounding problems to~\eqref{eq:P4}, and 2) lower-bounding
problems to~\eqref{eq:P4}, until a stopping criteria is met. In each
iteration $k$ of this process, we construct an upper-bounding
problem~\eqref{eq:UBPk} through McCormick relaxations of the bi-linear
terms $\{\nu^{(l)} \circ {\rho}^{(l)}\}_{l \in\mathcal{L}}$. This
upper bounding problem is a mixed-integer linear program and its
solution  gives the upper bound of~\eqref{eq:P4} and candidate
variable speed limits $x^{(k)}$. These $x^{(k)}$ can be used to
construct sample trajectories $\{{\rho}^{(l,k)}\}_{l \in
  \mathcal{L}}$ and a linear lower-bounding problem~\eqref{eq:LBPk}
for potential feasible solutions of~\eqref{eq:P4}.
\begin{algorithm}
\caption{Integer solution search algorithm} \label{Alg:issa}
\begin{algorithmic}[1] 
  \State Initialize $k=0$ \Repeat \State $k \leftarrow k+1$ \State
  Solve Problem~\eqref{eq:UBPk}, \Return $x^{(k)}$ and $\UB_k$ \State
  Generate sample trajectories $\{{\rho}^{(l,k)}\}_{l \in
    \mathcal{L}}$ \State Solve Problem~\eqref{eq:LBPk}, \Return
  $\obj_{k}$ and $\LB_k$ \Until{${\UB}_k -{\LB}_k \leq \epsilon$,
    or~\eqref{eq:UBPk} is infeasible, or a satisfactory suboptimal solution is found after certain running time $\subscr{T}{run}$}\\
  \Return data driven solution $\subscr{u}{best}:=u^{(q)}$ with
  certificate $\hat{J}(u^{(q)})$ such that $q \in
  \argmax_{p=1,\ldots,k} \{ {\obj}_p\}$
\end{algorithmic} 
\end{algorithm}

\textit{Upper-bounding Problems:} At each iteration $k$, the
upper-bounding problem~\eqref{eq:UBPk} is constructed using two extra
ingredients: 1) a McCormick relaxation of the bi-linear terms
$\{\nu^{(l)}_e(t){\rho}^{(l)}_e(t)\}_{e\in\mathcal{E},t\in\mathcal{T},l\in\mathcal{L}}$
in the objective function of~\eqref{eq:P4}, and 2) canonical integer
cuts that exclude the previous visited candidate variable speed limits
$\{ x^{(p)} \}_{p=1}^{k-1}$.

1) The McCormick envelope~\cite{MG:76} provides relaxations
of bi-linear terms, which is stated in the
following remark:
\begin{remark}[McCormick envelope]
  Consider two variables $x, \; y\; \in \real$ with upper and lower
  bounds, $\underline{x} \leq x \leq \overline{x}$, $\underline{y}
  \leq y \leq \overline{y}$. The McCormick envelope of the variable
  $s:=xy \in \real$ is characterized by the following constraints:
\begin{equation*}
  \begin{aligned}
    s\geq \overline{x}y+ x\overline{y}-\overline{x}\overline{y},
    \quad&
    s\geq \underline{x}y+x\underline{y}-\underline{x}\underline{y}, \\
    s\leq
    \overline{x}y+x\underline{y}-\overline{x}\underline{y},\quad&s\leq
    \underline{x}y+x\overline{y}-\underline{x}\overline{y}.
  \end{aligned}
\end{equation*}
\end{remark}
To construct a McCormick envelope for~\eqref{eq:UBPk}, let us denote
$\overline{\nu}_e:= \overline{u}_{e} \left(T^{-1} +
  \overline{\rho}_{e}\overline{\eta} \right)$ with
$e\in\mathcal{E}$. For each $e \in\mathcal{E}$, $t \in\mathcal{T}$ and
$l \in\mathcal{L}$, we have $0 \leq \nu^{(l)}_e(t) \leq
\overline{\nu}_e$, $0 \leq {\rho}^{(l)}_e(t) \leq
\overline{\rho}_e$, and the McCormick envelope of $s^{(l)}_e(t):=
\nu^{(l)}_e(t){\rho}^{(l)}_e(t)$ is given by
\begin{equation}	
	\begin{aligned}
          s^{(l)}_e(t) & \geq \overline{\nu}_e {\rho}^{(l)}_e(t)+
          \nu^{(l)}_e(t)\overline{\rho}_e-\overline{\nu}_e
          \overline{\rho}_e, \\
          s^{(l)}_e(t) & \geq 0, \\
          s^{(l)}_e(t)&\leq \overline{\nu}_e {\rho}^{(l)}_e(t),\\
          s^{(l)}_e(t)&\leq \nu^{(l)}_e(t)\overline{\rho}_e.
	\end{aligned} \label{eq:McR}
\end{equation}

2) The canonical integer cuts prevent~\eqref{eq:UBPk} from choosing
examined candidate variable speed limits $\{ x^{(p)}
\}_{p=1}^{k-1}$. Let $\Omega^{(p)}:=\setdef{(e,i) \in \mathcal{E}
  \times \mathcal{O}}{x^{(p)}_{e,i}=1}$ denote the index set of $x$
for which the value $x^{(p)}_{e,i}$ is $1$ at the previous iteration
$p$. Let $c^{(p)}:=|\Omega^{(p)}|$ denote the cardinality of the set
$\Omega^{(p)}$ and let $\overline{\Omega}^{(p)}:= \left(\mathcal{E}
  \times \mathcal{O}\right)\setminus \Omega^{(p)}$ denote the
complement of $\Omega^{(p)}$. The canonical integer cuts of
Problem~\eqref{eq:UBPk} at iteration $k$ are given by:
\begin{equation}
  \begin{aligned}
    &\sum_{ (e,i) \in \Omega^{(p)} } x_{e,i} - 
    \sum_{ (e,i) \in \overline{\Omega}^{(p)} } x_{e,i} \leq c^{(p)} -1, \\
    &\hspace{4cm} \forall p \in \untilone{k-1}.
	\end{aligned} \label{eq:cut}
\end{equation}

At each iteration $k$, the upper-bounding problem~\eqref{eq:UBPk} has
the following form: {\leqnomode
\begin{align} 
  \label{eq:UBPk} \tag{{UBP}$_k$}~&\max\limits_{\substack{x,y,z,s,
      {\rho}, \\ \lambda,\mu,\nu,\eta }}
  -\lambda \epsilon(\beta) - \frac{1}{N}\sum_{e,t,l}{\overline{f}_e\overline{\rho}_e\eta_e^{(l)}(t)  +s^{(l)}_e(t)}, \\
	  \st \quad & {\textbf{\small speed limits}}\; \eqref{eq:ueNew}, \; {\textbf{\small sample trajectories }}  \{\eqref{eq:G}, \; \eqref{eq:y}, \; \eqref{eq:RhohatNew2} \}  \nonumber\\
  &  {\textbf{\small no congestion}} \;  \{\eqref{eq:z},\; \eqref{eq:dual1}, \;\eqref{eq:dual2}, \;\eqref{eq:dual3}, \;\eqref{eq:dual4}\}, \nonumber \\
	&  {\textbf{\small McCormick envelope }} \;\eqref{eq:McR}, \; {\textbf{\small integer cuts}} \; \eqref{eq:cut}. \nonumber
\end{align}} 
We denote by $\UB_k$ the optimal objective value
of~\eqref{eq:UBPk} and $\UB_k$ is an upper bound of the original
nonconvex problem~\eqref{eq:P4}. We denote by $x^{(k)}$ the integer
part of the optimizers of~\eqref{eq:UBPk} and use it as a candidate
speed limit in the lower-bounding problem LBP$_k$ of~\eqref{eq:P4}.

\textit{Lower-bounding Problems:} To exploit the structure
of~\eqref{eq:P4} and find lower-bounding problems, let us define the
set $\Phi(x):=\setdef{(z,\lambda,\mu,\nu,\eta)}{\textbf{no congestion}
}$, $\Psi(x):=\setdef{(y,{\rho})}{\textbf{sample
    trajectories} 
}$ and $X:=\setdef{x}{\textbf{speed limits} 
}$. Problem~\eqref{eq:P4} can be equivalently written as: {\leqnomode
  \begin{align}  ~\max\limits_{\substack{x,y,z, {\rho}, \\
        \lambda,\mu,\nu,\eta }} & -\lambda \epsilon(\beta) -
    \frac{1}{N}\sum_{e,t,l}{\overline{f}_e
      \overline{\rho}_e\eta_e^{(l)}(t)  +\nu^{(l)}_e(t){\rho}^{(l)}_e(t)}, \nonumber\\
    \st \quad & (z,\lambda,\mu,\nu,\eta) \in \Phi(x), \;
    (y,{\rho})\in\Psi(x),\; x\in X. \nonumber
  \end{align}} Given $x^{(k)}\in X$ solved by~\eqref{eq:UBPk} at
iteration $k$, we have a candidate speed limit $u^{(k)}:={\sum_{i\in
    \mathcal{O}} \gamma^{(i)}
  (x^{(k)}_{1,i},\ldots,x^{(k)}_{n,i})}$. For each $l\in\mathcal{L}$
with given $u^{(k)}$, the sample trajectory ${\rho}^{(l)}$ is uniquely
determined by $(\rho^{(l)}(0), \omega^{(l)})$, via the uniqueness
solution of the linear time-invariant systems. Therefore, the element
$(y,{\rho}) \in\Psi(x^{(k)})$ is unique.  Using the constraints set
$\Psi(x^{(k)})$, we then construct the unique sample trajectories
$\{{\rho}^{(l,k)}\}_{l \in \mathcal{L}}$. The unique sample
trajectories enable us to define the linear lower bounding problem at
iteration $k$, as follows: {\leqnomode
\begin{align}
  \label{eq:LBPk} \tag{{LBP}$_k$}~\max\limits_{\substack{z,
      \lambda,\mu,\nu,\eta }} &-\lambda \epsilon(\beta) -
  \frac{1}{N}\sum_{e,t,l}{\overline{f}_e
    \overline{\rho}_e\eta_e^{(l)}(t)  +\nu^{(l)}_e(t){\rho}^{(l,k)}_e(t)}, \\
  \st \quad & (z,\lambda,\mu,\nu,\eta) \in \Phi(x^{(k)}). \nonumber
\end{align}} Let $\obj_{k}$ 
denote the optimal objective value
of~\eqref{eq:LBPk}. If Problem~\eqref{eq:LBPk} is solved to optimum
with a finite $\obj_{k}$, we then obtain a feasible solution
of~\eqref{eq:P4} with speed limit $u^{(k)}:={\sum_{i\in \mathcal{O}}
  \gamma^{(i)} (x^{(k)}_{1,i},\ldots,x^{(k)}_{n,i})}$ and certificate
$\hat{J}(u^{(k)}):={\obj}_{k}$. Otherwise, Problem~\eqref{eq:LBPk}
is either infeasible or unbounded and we let $\obj_{k}=-\infty$. The
lower bound of~\eqref{eq:P4} is then calculated by $\LB_k=
\max_{p=1,\ldots,k}\{ {\obj}_{p} \}$. The stopping criteria of the
algorithm can be determined by 1) ${\UB}_k -{\LB}_k \leq
\epsilon$, or 2)~\eqref{eq:UBPk} is infeasible, or 3) a satisfactory
suboptimal solution is found after certain running time
$\subscr{T}{run}$.  We refer to~\cite{LX-TA-BP:11} for the finite
convergence of Algorithm~\ref{Alg:issa} to a global $\epsilon$-optimal
solution using both the first and second stopping criteria. To find a
potentially good feasible solution within certain running time
$\subscr{T}{run}$, we further propose the third criteria. A
satisfactory suboptimal solution after running time $\subscr{T}{run}$
is then a feasible solution that achieves the lower bound of the
algorithm. If no feasible solution is found within time
$\subscr{T}{run}$, we wait until a feasible solution is obtained.
	
\vspace*{-0.25cm}
\section{Simulations}\label{sec:Sim}
In this section, we demonstrate in an example how to find a solution
to~\eqref{eq:P4} that results in a data-driven variable-speed limit $u
\in \real^5$ with performance guarantee~\eqref{eq:perfnew}. We
consider a highway with length $L=10\rm{km}$ and we divide it into
$n=5$ segments. Let the unit size of each time slot $\delta=30\sec$
and consider $T=20$ time slots for a $10$min planning horizon.  For
each edge $e \in \mathcal{E}$, we assume a jam density of
$\overline{\rho}_e=1050\rm{vec/km}$\footnote{The unit ``vec'' stands
  for ``vehicles''.}, a capacity of $\overline{f}_e=3.1\times
10^4\rm{vec/h}$ and a maximal free flow of
$\overline{u}_e=140\rm{km/h}$. Let us consider $m=5$ different
candidate speed limits $\Gamma=\{40{\rm{km/h}},\;
60{\rm{km/h}},\;80{\rm{km/h}},\;100{\rm{km/h}},\;120{\rm{km/h}}
\}$. On the $\supscr{4}{th}$ edge $e:=(3,4) \in \mathcal{E}$, we
assume an accident happens during $\mathcal{T}$ with parameters
$\supscr{f}{U}_e=2.7\times 10^4\rm{vec/h}$ and
$\supscr{\rho}{U}_e=\overline{\rho}_e$. To evaluate the effect of the
proposed algorithm, samples of the random variables $w$ and $\rho(0)$
are needed.
In real-case studies, samples $\{\rho^{(l)}(0)\}_{l\in\mathcal{L}}$
can be obtained from road sensors (loop detectors), while samples of
the uncertain flows $\{\omega^{(l)}\}_{l\in\mathcal{L}}$ can be
constructed either from a database of flow data on the road, or from
the current measurements of ramp flows with the assumption that the
stochastic process $\{\omega(t)\}_{t\in\mathcal{T}}$ is stationary.

In this simulation example, the index set of accessible samples is given by $\mathcal{L}=\{1,2,3\}$. For each $l\in\mathcal{L}$, let us assume
that each segment $e\in\mathcal{E}$ initially operates under a free
flow condition with an initial density
$\rho^{(l)}_e(0)=260\rm{vec/km}$. For each edge
$e\in\mathcal{E}\setminus \{ 1\}$ and time $t\in\mathcal{T}$, we will
assume that samples $\{\omega^{(l)}_e(t)\}_{l\in\mathcal{L}}$ are
generated from a uniform distribution within interval
$[-1500,2500]\rm{vec/h}$. 
To ensure significant inflows of the system, we further let the
samples $\{\omega^{(l)}_1(t)\}_{l\in\mathcal{L}}$ of the first segment
to be chosen from the uniform distribution within interval $[2\times
10^4,2.4\times 10^4]\rm{vec/h}$. 
	We also let the confidence level be $\beta=0.95$ and the radius
of the Wasserstein Ball $\epsilon(\beta)=0.985$ as calculated
in~\cite{DL-SM:18-extended}.

To generate feasible solutions that can be carried out for a real time
transportation system, we allocate $\subscr{T}{run}=5$min execution
time to the proposed Algorithm~\ref{Alg:issa}, and run it on a machine
with $3.4$GHz CPU and $4$G RAM. In $5$ minutes, the algorithm computed
$5$ feasible candidate speed limits and discarded $13$ infeasible
candidate speed limits. The feasible solutions were obtained after
$120\sec$, $138\sec$, $174\sec$, $189\sec$ and
$270\sec$, 
respectively. 
We verified that $\hat{J}(u^{(3)})=2.435\times 10^4\rm{vec/h}$ is the
highest certificate obtained, i.e., $\hat{J}(u^{(3)}) \in
\argmax_{p=1,\ldots,5} \{ \hat{J}(u^{(p)}) \;| \; u^{(p)}
  \text{ is feasible} \}$, 
	and the desired speed
limits are $u^{(3)}=[100,120,100,80,120]{\rm{km/h}}$. The algorithm
terminated at iteration $k=18$, with bounds $\LB_{k}= \hat{J}(u^{(3)})$ 
and $\UB_{k}=9.0 \times
10^7\rm{vec/h}$. It can be seen that the upper bound of the algorithm
is loose, but the implementable solutions can be obtained in
reasonable computational time. With knowledge of the underlying
distribution, we see that the value of the certificate averaged on segments, given by
$\hat{J}(u^{(3)})/5$, 
	is higher than the upper
bound of the random flows injected in the first segment of the
highway. This indicates that, with $95\%$ confidence, the speed limit
$u^{(3)}$ guarantees no congestion flows along the highway although
initially the highway is congested.

To evaluate the out-of-sample performance of the speed limits
$u^{(3)}$, we generated $\subscr{N}{val}= 10^3$ validation samples of
$(\omega, \rho(0))$ 
and simulated the cell transmission
model~\cite{CFD:94} with the same parameter settings but a $30$min
time horizon.
Fig.~\ref{fig:result_density} 
shows the average of the sample trajectories over time, i.e., the
function $\frac{1}{\subscr{N}{val}}\sum_{l\in
  \untilone{\subscr{N}{val}}} {\rho}_e^{(l)}(t)$ for each segment $e$,
with and without speed limits. For the density evolution with speed
limits $u^{(3)}$, we verified that the density trajectory of accident
edge $(4)$ did not exceed its critical density
$\rho_4^{c}(80\rm{km/h})=335\rm{vec/km}$ and thus the road
$\mathcal{G}$ kept free of congestion in this planning horizon
$\mathcal{T}$. However, for the density evolution without speed
limits, vehicles were accumulated on edge $(4)$ and the congestion was
propagated along edges of the road $\mathcal{G}$. 

\begin{figure}[tbp]%
\centering
\includegraphics[width=0.5\textwidth]{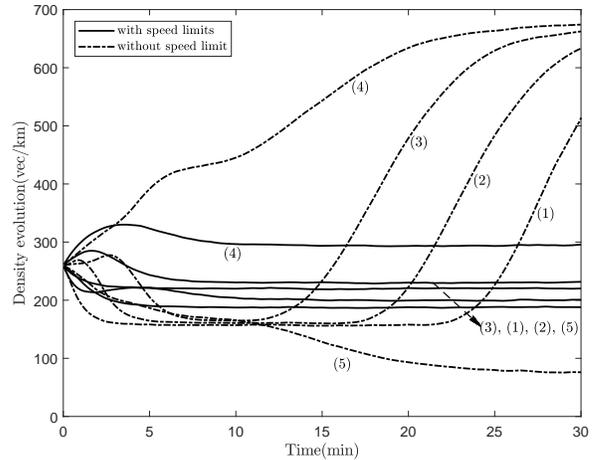}%
\caption{\small Density evolution of each segment $e$, with and
  without speed limits $u^{(3)}$. Each trajectory corresponds to a
  segment $e\in \{(1),(2),\ldots,(5) \}$. The pointed down arrow
  indicates that from top to bottom the density evolutions are for
  segments $(3)$, $(1)$,$(2)$ and $(5)$, respectively.}%
\label{fig:result_density}%
\end{figure}

\section{Conclusions}\label{sec:Conclude}
In this paper, we proposed a traffic model that considered uncertain
inflow, outflow, and random events along a highway. We then formulate
a control problem in form of~\eqref{eq:P}, where realizations of
the unknown inflows and outflows are employed to derive data-driven
variable speed limits that have guaranteed out-of-sample
performance. We achieved this by adopting DRO theory to the equivalent
Problem~\eqref{eq:P1}, which further results into the mix-integer
bi-linear Problem~\eqref{eq:P4}. Problem~\eqref{eq:P4} is solved by
means of a proposed integer-solution search algorithm that is derived from decomposition-based method. 
	The focus of our current work is on
considering more complex traffic networks and the use of moving
horizons to derive data-driven variable speed limits by leveraging
real-time dynamic data.

 \bibliographystyle{IEEEtran}
%\bibliography{../../bib/alias,../../bib/SMD-add,../../bib/SM,../../bib/Main,../../bib/Main-sonia}
\bibliography{alias,SMD-add,SM}

%%%%%%%%%%%%%%%%%%%%%%%%%%%%%%%%%%%%%%%%%%%%%%%%%%%%%%%%%%%%%%%%%%%%%%
\end{document}